\documentclass[11pt]{amsart}
\usepackage{amssymb, amsmath}
\usepackage{verbatim}
\usepackage{enumitem}
\usepackage{hyperref}
\numberwithin{equation}{section}

\DeclareMathOperator{\ord}{ord}
\DeclareMathOperator{\Res}{Res}

\newcommand{\norm}[1]{\left\lVert#1\right\rVert}

\theoremstyle{plain} 
\newtheorem{theorem}{Theorem}[section] 
\newtheorem{lemma}[theorem]{Lemma}     

\newtheorem{proposition}[theorem]{Proposition}

\theoremstyle{definition} 

\newtheorem{example}[theorem]{Example}
\newtheorem{algorithm}[theorem]{Algorithm}

\theoremstyle{remark} 
\newtheorem{remark}[theorem]{Remark}
 
\newtheorem*{acknowledgments}{Acknowledgments}

\begin{document}

\title{Computing the Canonical Height of a Point in Projective Space}

\author[]{Elliot Wells}
\address{Mathematics Department, Brown University, Providence, RI 02912}
\email{ellwells@math.brown.edu}
\date{February 16, 2016}

\begin{abstract}
We give an algorithm which requires no integer factorization for computing the canonical height of a point in $\mathbb{P}^{1}(\mathbb{Q})$ relative to a morphism $\phi: \mathbb{P}_{\mathbb{Q}}^{1} \rightarrow \mathbb{P}_{\mathbb{Q}}^{1}$ of degree $d \geq 2$.
\end{abstract}

\maketitle

\section{Introduction}

Let $\phi: \mathbb{P}_{\mathbb{Q}}^{1} \rightarrow \mathbb{P}_{\mathbb{Q}}^{1}$ be a morphism of degree $d \geq 2$, and let $h$ be the logarithmic height on $\mathbb{P}^{1}(\mathbb{Q})$. The \textit{canonical height function relative to $\phi$} (see \cite{Silverman:2007}) is the function
\[
\hat{h}_{\phi}: \mathbb{P}^{1}(\mathbb{Q}) \longrightarrow \mathbb{R}_{\geq 0}
\]
defined by
\begin{equation}\label{eq:height_definition}
\hat{h}_{\phi}(P) = \lim_{n \rightarrow \infty}d^{-n}h(\phi^{n}(P));
\end{equation}
it is uniquely characterized by the two properties
\[
\hat{h}_{\phi}(\phi(P)) = d\hat{h}_{\phi}(P) \quad\text{and}\quad \hat{h}_{\phi}(P) = h(P) + O(1),
\]
where the constant in the $O(1)$ term depends on $\phi$ but not on $P$. The canonical height of a point $P$ relative to $\phi$ gives information about the behavior of $P$ under the iteration of $\phi$, and so canonical heights have numerous applications in arithmetic dynamics and arithmetic geometry. For example, $\hat{h}_{\phi}(P) = 0$ if and only if $P$ is preperiodic under $\phi$, and this fact provides a quick proof of a result of Northcott \cite{Northcott} that $\phi$ has finitely many preperiodic points in $\mathbb{P}^{1}(\mathbb{Q})$. Additionally, one of the quantities appearing in the Birch and Swinnerton-Dyer Conjecture -- the regulator -- is defined in terms of canonical heights on elliptic curves, which equivalently are canonical heights for Latt\`{e}s maps on $\mathbb{P}^{1}$. There are also a number of related conjectures in arithmetic dynamics, such as the Dynamical Lehmer Conjecture, concerning lower bounds on canonical heights of non-preperiodic points (see \cite[Conjecture 3.25]{Silverman:2007}). 

In this note, we give an algorithm for computing $\hat{h}_{\phi}(P)$ that is efficient even if every lift $\Phi = [F, G]: \mathbb{A}^{2}(\mathbb{Q}) \rightarrow \mathbb{A}^{2}(\mathbb{Q})$ of $\phi$ has large integer coefficients or if $d$ is large. To appreciate why such an algorithm might be helpful, let us recall the usual methods for computing $\hat{h}_{\phi}(P)$. The limit definition \eqref{eq:height_definition} of $\hat{h}_{\phi}$ is generally unsuitable for computation, because the number of digits of the coordinates of $\phi^{n}(P)$ grows exponentially with $n$. As an alternative, one decomposes $\hat{h}_{\phi}(P)$ (with $P \neq \infty$) as a sum of local canonical heights corresponding to the different absolute values on $\mathbb{Q}$:
\begin{equation}\label{eq:local_heights}
\hat{h}_{\phi}(P) = \sum_{v \in M_{\mathbb{Q}}}\hat{\lambda}_{\phi, v}(P) = \hat{\lambda}_{\phi, \infty}(P) + \sum_{p \ \text{prime}}\hat{\lambda}_{\phi, p}(P),
\end{equation}
see \cite[Theorem 5.61]{Silverman:2007}. Let $\Phi = [F, G]: \mathbb{A}^{2}(\mathbb{Q}) \rightarrow \mathbb{A}^{2}(\mathbb{Q})$ be a lift of $\phi$ such that $F$ and $G$ have relatively prime integer coefficients, let $R := \Res(F, G)$ be the resultant of $F$ and $G$, and let $\norm{F, G}$ denote the maximum of the absolute values of the coefficients of $F$ and $G$. Then for all primes $p$ such that $p \nmid R$, the local canonical height of $P = [x, y]$ at $p$ is given by the simple formula
\[
\hat{\lambda}_{\phi, p}(P) = \log\left(\frac{\max\{|x|_{p}, |y|_{p}\}}{|y|_{p}}\right),
\]
which allows us to rewrite \eqref{eq:local_heights} as
\begin{equation}\label{eq:modified_local_heights}
\hat{h}_{\phi}(P) = h(P) + \tilde{\lambda}_{\phi, \infty}(P) + \sum_{p \mid R}\tilde{\lambda}_{\phi, p}(P),
\end{equation}
where
\begin{equation}\label{eq:modified_local_height}
\tilde{\lambda}_{\phi, v}(P) := \hat{\lambda}_{\phi, v}(P) - \log\left(\frac{\max\{|x|_{v}, |y|_{v}\}}{|y|_{v}}\right).
\end{equation}
For any fixed $v \in M_{\mathbb{Q}}$, we can efficiently approximate the value of $\hat{\lambda}_{\phi, v}(P)$ (and hence also $\tilde{\lambda}_{\phi, v}(P)$) with the method described in \cite[Section 5]{Call:1993} (or with the algorithm in \cite[exercise 5.29]{Silverman:2007}). So using decomposition \eqref{eq:modified_local_heights}, we can efficiently compute an approximation of $\hat{h}_{\phi}(P)$, \textit{provided that we know the primes dividing $R$}. In practice, however, factoring $R$ is often time-consuming. Indeed, for most morphisms $\phi$ we have $R \approx \norm{F, G}^{2d}$, so $R$ grows exponentially with $d$ and can thus be time-consuming to factor even for moderately sized $d$ (in addition to when $F$ and/or $G$ has large coefficients). In short, for generic $\phi$ we expect factoring $R$ to be a non-trivial computational task.

In this note, we describe a practical algorithm \textit{which requires no integer factorization} for computing an approximation of the nonarchimedean term in \eqref{eq:modified_local_heights}. Combining this with any already existing algorithm for approximating the value of the archimedean term (e.g., \cite[Section 5]{Call:1993} or \cite[exercise 5.29]{Silverman:2007}) yields an algorithm for approximating $\hat{h}_{\phi}(P)$. The inspiration behind our algorithm comes from an algorithm (requring no integer factorization) in \cite{Stoll:quasilinear} for computing the canonical height of a point on an elliptic curve; we show how the ideas in \cite{Stoll:quasilinear} -- which deal with morphisms of elliptic curves over $\mathbb{Q}$ -- generalize to morphisms of $\mathbb{P}^{1}$ over $\mathbb{Q}$.

In fact, it is instructive to compare our original problem of computing $\hat{h}_{\phi}(P)$ to that of computing the canonical height of a point on an elliptic curve (as defined, e.g., in \cite{Silverman:1994}). On the one hand, the latter computation can be viewed as a special case of the former. Indeed, the duplication map $Q \mapsto 2[Q]$ on an elliptic curve $E/\mathbb{Q}$ induces a morphism $\psi: \mathbb{P}_{\mathbb{Q}}^{1} \rightarrow \mathbb{P}_{\mathbb{Q}}^{1}$, called a Latt\`{e}s map, with the property that for any $P \in E(\mathbb{Q})$ the canonical height of $P$ (in the sense of the definition for points on an elliptic curve) equals $\hat{h}_{\psi}(P)$, after appropriate normalizations. For a more detailed explanation of Latt\`{e}s maps, see, e.g., \cite[Section 6.4]{Silverman:2007}. However, computing canonical heights of points on an elliptic curve $E/\mathbb{Q}$ is in some sense much easier than computing $\hat{h}_{\phi}(P)$ for an arbitrary morphism $\phi: \mathbb{P}_{\mathbb{Q}}^{1} \rightarrow \mathbb{P}_{\mathbb{Q}}^{1}$, due to the extra structure arising from the group law on $E$. More precisely, one can show that for $P \in E(\mathbb{Q})$, the nonarchimedean term in \eqref{eq:local_heights} has the form
\[
\sum_{p}r_{p}\log(p)
\]
where the $r_{p}$, which \textit{a priori} are arbitrary real numbers, are in fact rational numbers of a very precise form; moreover, these rational numbers can often be determined, with very little factorization, by computing just a few multiples $[2]P, [3]P, \ldots$ of $P$. These observations are made in \cite{Silverman:1997} to give a nearly factorization free algorithm for computing canonical heights on elliptic curves, and \cite{Stoll:quasilinear} gives a completely factorization free algorithm by exploiting the constraints on the values of the $r_{p}$'s. For a morphism $\phi: \mathbb{P}_{\mathbb{Q}}^{1} \rightarrow \mathbb{P}_{\mathbb{Q}}^{1}$, in contrast, it is unknown whether the $r_{p}$'s must be rational, and in any case they cannot be determined simply by computing a small number of iterates of $\phi$. This makes computing $\hat{h}_{\phi}(P)$ more difficult, and any algorithm for doing so will, in general, return an approximation of the $r_{p}$'s, rather than an exact value as provided by the algorithms of \cite{Silverman:1997} and \cite{Stoll:quasilinear}.

Our paper is organized as follows. In Section \ref{s:notation}, we establish notation, describe a decomposition of $\hat{h}_{\phi}(P)$ analogous to \eqref{eq:modified_local_heights}, and prove some basic results that are needed to analyze our algorithm. In section \ref{s:algorithm}, we describe a factorization free algorithm for computing the nonarchimedean term in our decomposition. Finally, we present some examples in section \ref{s:examples}.

\begin{acknowledgments}
The author would like to thank his advisor, Joseph Silverman for his guidance and insightful discussions on this problem, as well as for his helpful suggestions on improving the drafts of this paper.
\end{acknowledgments}

\section{Notation and Preliminary Results}\label{s:notation}

Below is a summary of the notation and definitions used throughout this note:
\begin{itemize}
	\item $M_{\mathbb{Q}}$ - the set of absolute values on $\mathbb{Q}$, with the usual normalizations.
	\item $M^{\infty}_{\mathbb{Q}}$ - the set of archimedean absolute values on $\mathbb{Q}$.
	\item $M^{0}_{\mathbb{Q}}$ - the set of nonarchimedean absolute values on $\mathbb{Q}$.
	\item $\ord_{p}(a)$ - the greatest exponent $e$ such that $p^{e}$ divides $a \in \mathbb{Z}$, where $p$ is a prime.
	\item $\phi: \mathbb{P}_{\mathbb{Q}}^{1} \rightarrow \mathbb{P}_{\mathbb{Q}}^{1}$ - a morphism of degree $d \geq 2$.
	\item $\Phi = [F, G]: \mathbb{A}^{2}(\mathbb{Q}) \rightarrow \mathbb{A}^{2}(\mathbb{Q})$ - a lift of $\phi$ with integer coefficients.
	\item $\norm{F, G}$ - the maximum of the (archimedean) absolute values of the coefficients of $F$ and $G$.
	\item $\norm{\ \cdot \ }_{v}$ - the sup norm associated to $v \in M_{\mathbb{Q}}$; i.e., 
	\[
	\norm{(x, y)}_{v} = \max\{|x|_{v}, |y|_{v}\}
	\]
	for any $(x, y) \in \mathbb{A}^{2}(\mathbb{Q})$. 
	\item $\Res(F, G)$ - the resultant of the homogeneous polynomials $F$ and $G$.
	\item $h$ - the logarithmic height on $\mathbb{P}^{1}(\mathbb{Q})$.
	\item $\hat{h}_{\phi}$ - the canonical height associated with the morphism $\phi$.
	\item $\Lambda_{\Phi, v}: \mathbb{P}^{1}(\mathbb{Q}) \rightarrow \mathbb{R}$ - the function defined by
	\[
	\Lambda_{\Phi, v}: [x, y] \mapsto \log(\norm{\Phi(x, y)}_{v}^{-1}) - d\log(\norm{(x, y)}_{v}^{-1})
	\]
	for $v \in M_{\mathbb{Q}}$ and a lift $\Phi = [F, G]$ of $\phi$.
	\item $\Omega_{\Phi, s}, \mathcal{H}_{\Phi, s}: \mathbb{P}^{1}(\mathbb{Q}) \rightarrow \mathbb{R}$ - for $s \in \{0, \infty\}$, the functions defined by
	\[
	\Omega_{\Phi, s}: P \mapsto \sum_{v \in M^{s}_{\mathbb{Q}}}\Lambda_{\Phi, v}(P)
	\]
	and
	\begin{equation}\label{eq:decomposition_terms}
	\mathcal{H}_{\Phi, s}: P \mapsto \sum_{n=0}^{\infty}\frac{\Omega_{\Phi, s}(\phi^{n}(P))}{d^{n+1}}
	\end{equation}
	for a lift $\Phi = [F, G]$ of $\phi$.
\end{itemize}

\begin{remark}
For any $P \in \mathbb{P}^{1}(\mathbb{Q})$, we have
\[
\mathcal{H}_{\Phi, \infty}(P) = \tilde{\lambda}_{\phi, \infty}(P),
\]
where $\tilde{\lambda}_{\phi, v}(P)$ is the modified local canonical height defined by \eqref{eq:modified_local_height}. In particular, there are efficient algorithms in the literature for computing $\mathcal{H}_{\Phi, \infty}(P)$.
\end{remark}

We begin by obtaining a decomposition of $\hat{h}_{\phi}(P)$, analogous to \eqref{eq:modified_local_heights}, as a sum of the logarithmic height of $P$ and an archimedean and nonarchimedean term.

\begin{lemma}\label{l:height_identity}
Let $\phi: \mathbb{P}_{\mathbb{Q}}^{1} \rightarrow \mathbb{P}_{\mathbb{Q}}^{1}$ be a morphism of degree $d \geq 2$, and let $\Phi$ be a lift of $\phi$. Then for any $P \in \mathbb{P}^{1}(\mathbb{Q})$, we have
\[
h(\phi(P)) - dh(P) = -(\Omega_{\Phi, \infty}(P) + \Omega_{\Phi, 0}(P)).
\]
\end{lemma}
\begin{proof}
Fix $P = [x, y] \in \mathbb{P}^{1}(\mathbb{Q})$. We have
\begin{align*}
h(\phi(P)) - dh(P) &= \sum_{v \in M_{\mathbb{Q}}}-\log(\norm{\Phi(x, y)}_{v}^{-1}) - d\sum_{v \in M_{\mathbb{Q}}}-\log(\norm{(x, y)}_{v}^{-1}) \\
								   &= \sum_{v \in M_{\mathbb{Q}}}-(\log(\norm{\Phi(x, y)}_{v}^{-1}) - d\log(\norm{(x, y)}_{v}^{-1})) \\
								   &= \sum_{v \in M_{\mathbb{Q}}}-\Lambda_{\Phi, v}(P) \\
								   &= \sum_{v \in M^{\infty}_{\mathbb{Q}}}-\Lambda_{\Phi, v}(P) + \sum_{v \in M^{0}_{\mathbb{Q}}}-\Lambda_{\Phi, v}(P) \\
								   &= -(\Omega_{\Phi, \infty}(P) + \Omega_{\Phi, 0}(P)).
\end{align*}
\end{proof}

\begin{proposition}\label{p:canonical_height_formula}
Let $\phi: \mathbb{P}_{\mathbb{Q}}^{1} \rightarrow \mathbb{P}_{\mathbb{Q}}^{1}$ be a morphism of degree $d \geq 2$, and let $\Phi$ be a lift of $\phi$. Then for any $P \in \mathbb{P}^{1}(\mathbb{Q})$, we have
\[
\hat{h}_{\phi}(P) = h(P) - \mathcal{H}_{\Phi, \infty}(P) - \mathcal{H}_{\Phi, 0}(P).
\]
\end{proposition}
\begin{proof}
Fix $P \in \mathbb{P}^{1}(\mathbb{Q})$. From the definition \eqref{eq:height_definition} of $\hat{h}_{\phi}(P)$ and Lemma \ref{l:height_identity}, we have
\begin{align*}
\hat{h}_{\phi}(P) &= \lim_{m \rightarrow \infty}\frac{h(\phi^{m}(P))}{d^{m}} \\
							    &= \lim_{m \rightarrow \infty}\left(h(P) + \sum_{n = 0}^{m-1}\frac{h(\phi^{n+1}(P)) - dh(\phi^{n}(P))}{d^{n+1}}\right) \\
                  &= h(P) + \sum_{n=0}^{\infty}\frac{-(\Omega_{\Phi, \infty}(\phi^{n}(P)) + \Omega_{\Phi, 0}(\phi^{n}(P)))}{d^{n+1}} \\
                  &= h(P) - \sum_{n=0}^{\infty}\frac{\Omega_{\Phi, \infty}(\phi^{n}(P))}{d^{n+1}} - \sum_{n=0}^{\infty}\frac{\Omega_{\Phi, 0}(\phi^{n}(P))}{d^{n+1}} \\
                  &= h(P) - \mathcal{H}_{\Phi, \infty}(P) - \mathcal{H}_{\Phi, 0}(P).
\end{align*}
\end{proof}

Next, we establish some simple facts that are needed in the following section to analyze our algorithm. 

\begin{lemma}\label{l:compute_omega}
Let $\phi: \mathbb{P}_{\mathbb{Q}}^{1} \rightarrow \mathbb{P}_{\mathbb{Q}}^{1}$ be a morphism of degree $d \geq 2$, and let $\Phi = [F, G]$ be a lift of $\phi$ such that $F$ and $G$ have integer coefficients. Let $Q \in \mathbb{P}^{1}(\mathbb{Q})$, and write $Q = [x, y]$ with $(x, y) \in \mathbb{Z}^{2}$ and $\gcd(x, y) = 1$. Then 
\[
\Omega_{\Phi, 0}(Q) = \log(\gcd(F(x, y), G(x, y))).
\]
\end{lemma}
\begin{proof}
Let $v \in M^{0}_{\mathbb{Q}}$ be the absolute value associated with the prime number $p$. The assumption that $(x, y) \in \mathbb{Z}^{2}$ with $\gcd(x, y) = 1$ implies that 
\[
\log(\norm{(x, y)}_{v}^{-1}) = 0,
\] 
and it follows that
\begin{align*}
\Lambda_{\Phi, v}(Q) &= \log(\norm{\Phi(x, y)}_{v}^{-1}) - d\log(\norm{(x, y)}_{v}^{-1}) \\
									   &= \log(\norm{F(x, y), G(x, y)}_{v}^{-1}) \\
									   &= \log{p^{\min\{\ord_{p}(F(x, y)), \ord_{p}(G(x, y))\}}}.
\end{align*}
So we have
\begin{align*}
\Omega_{\Phi, 0}(Q) &= \sum_{v \in M^{0}_{\mathbb{Q}}}\Lambda_{\Phi, v}(Q) \\
									  &= \sum_{p \ \text{prime}}\log{p^{\min\{\ord_{p}(F(x, y)), \ord_{p}(G(x, y))\}}} \\
									  &= \log\left(\prod_{p \ \text{prime}}{p^{\min\{\ord_{p}(F(x, y)), \ord_{p}(G(x, y))\}}}\right) \\
									  &= \log(\gcd(F(x, y), G(x, y))).
\end{align*}
\end{proof}

\begin{lemma}\label{l:gcd_divides_res}
Let $F(X, Y), G(X, Y) \in \mathbb{Z}[X, Y]$ be homogeneous polynomials of degree $d$, and let $(a, b) \in \mathbb{Z}^{2}$ with $\gcd(a, b) = 1$. Then $\gcd(F(a, b), G(a, b))$ divides $\Res(F, G)$.
\end{lemma}
\begin{proof}
By \cite[Proposition 2.13(c)]{Silverman:2007}, there exist polynomials
\[
A_{1}, B_{1}, A_{2}, B_{2} \in \mathbb{Z}[X, Y]
\]
such that
\begin{align*}
A_{1}(X, Y)F(X, Y) + B_{1}(X, Y)G(X, Y) &= \Res(F, G)X^{2d-1}, \\
A_{2}(X, Y)F(X, Y) + B_{2}(X, Y)G(X, Y) &= \Res(F, G)Y^{2d-1},
\end{align*}
and evaluating at $(a, b)$ yields that
\begin{align*}
A_{1}(a, b)F(a, b) + B_{1}(a, b)G(a, b) &= \Res(F, G)a^{2d-1}, \\
A_{2}(a, b)F(a, b) + B_{2}(a, b)G(a, b) &= \Res(F, G)b^{2d-1}.
\end{align*}
So $\gcd(F(a, b), G(a, b))$ divides both $\Res(F, G)a^{2d-1}$ and $\Res(F, G)b^{2d-1}$ with $\gcd(a, b) = 1$, which implies that $\gcd(F(a, b), G(a, b))$ divides $\Res(F, G)$.
\end{proof}

\begin{lemma}\label{l:compute_omega_modulo}
Let $F(X, Y), G(X, Y) \in \mathbb{Z}[X, Y]$ be homogeneous polynomials of degree $d$, and let $(x, y) \in \mathbb{Z}^{2}$ with $\gcd(x, y) = 1$. Fix any $(x', y') \in \mathbb{Z}^{2}$ such that
\begin{align*}
x' &\equiv F(x, y) \pmod{\Res(F, G)}, \\
y' &\equiv G(x, y) \pmod{\Res(F, G)}.
\end{align*}
Then
\[
\gcd(F(x, y), G(x, y)) = \gcd(x', y', \Res(F, G)).
\]
\end{lemma}
\begin{proof}
In general, if $(a, b) \in \mathbb{Z}^{2}$ and $R \in \mathbb{Z}$ such that $\gcd(a, b)$ divides $R$, then
\[
\gcd(a, b) = \gcd(a + cR, b + dR, R)
\]
for any $c, d \in \mathbb{Z}$. By Lemma \ref{l:gcd_divides_res}, we can apply this to the case $a = F(x, y), b = G(x, y)$, and $R = \Res(F, G)$, which gives the desired result.
\end{proof}

\begin{lemma}\label{l:omega_bound}
Let $\phi: \mathbb{P}_{\mathbb{Q}}^{1} \rightarrow \mathbb{P}_{\mathbb{Q}}^{1}$ be a morphism of degree $d \geq 2$, and let $\Phi = [F, G]$ be a lift of $\phi$ such that $F$ and $G$ have integer coefficients. Then for any $P \in \mathbb{P}^{1}(\mathbb{Q})$, we have
\[
\left|\Omega_{\Phi, 0}(P)\right| \leq \log{\left|\Res(F, G)\right|}.
\]
\end{lemma}
\begin{proof}
Fix $P \in \mathbb{P}^{1}(\mathbb{Q})$, and write $P = [x, y]$ with $(x, y) \in \mathbb{Z}^{2}$ and $\gcd(x, y) = 1$. By Lemma \ref{l:compute_omega}, we know that
\[
\Omega_{\Phi, 0}(P) = \log(\gcd(F(x, y), G(x, y))),
\]
where $\gcd(F(x, y), G(x, y))$ divides $\Res(F, G)$ by Lemma \ref{l:gcd_divides_res}. In particular, we have
\[
\gcd(F(x, y), G(x, y)) \leq \left|\Res(F, G)\right|,
\]
which shows that
\[
\left|\Omega_{\Phi, 0}(P)\right| \leq \log{\left|\Res(F, G)\right|}.
\]
\end{proof}

\section{The Algorithm}\label{s:algorithm}

By Proposition \ref{p:canonical_height_formula}, to efficiently compute $\hat{h}_{\phi}(P)$, it suffices to efficiently compute the three quantities $h(P), \mathcal{H}_{\Phi, \infty}(P)$, and $\mathcal{H}_{\Phi, 0}(P)$. Of course, computing $h(P)$ is easy, and there are efficient algorithms for computing $\mathcal{H}_{\Phi, \infty}(P)$, e.g., the method described in \cite[Section 5]{Call:1993} or \cite[exercise 5.29]{Silverman:2007}.

We give an algorithm that efficiently approximates the value of $\mathcal{H}_{\Phi, 0}(P)$ by computing the first $N$ terms in the infinite series definition \eqref{eq:decomposition_terms} of $\mathcal{H}_{\Phi, 0}(P)$; notably, our algorithm does not require the prime factorization of $\Res(F, G)$.

\begin{algorithm}\label{a:algorithm}\

\textbf{Input}: A morphism $\phi: \mathbb{P}_{\mathbb{Q}}^{1} \rightarrow \mathbb{P}_{\mathbb{Q}}^{1}$ of degree $d \geq 2$ with lift $\Phi = [F, G]$ such that $F$ and $G$ have integer coefficients, a point $P \in \mathbb{P}^{1}(\mathbb{Q})$, and a number $N$ of terms in the sum to compute.

\textbf{Output}: An approximation of the value of $\mathcal{H}_{\Phi, 0}(P)$, accurate to within $O(d^{-N})$. More precisely, the output $\mathcal{H}$ satisfies
\[
\left|\mathcal{H}_{\Phi, 0}(P) - \mathcal{H}\right| \leq \frac{\log{\left|\Res(F, G)\right|}}{(d-1)d^{N}},
\]
and is computed by working with numbers of size at most $O(\Res(F, G)^{N})$.

\begin{enumerate}
	\item Write $P = [x_{0}, y_{0}]$ with $(x_{0}, y_{0}) \in \mathbb{Z}^{2}$ and $\gcd(x_{0}, y_{0}) = 1$.
	\item Set $\mathcal{H} = 0$ and $R = \Res(F, G)$.
	\item For $i = 0, 1, \ldots, N-1$:
\begin{enumerate}
	\item $\begin{aligned}[t]
	      x'_{i+1} &= F(x_{i}, y_{i}) \pmod{R^{N-i}}, \\
	      y'_{i+1} &= G(x_{i}, y_{i}) \pmod{R^{N-i}}.
	      \end{aligned}$
	\item $g_{i} = \gcd(x'_{i+1}, y'_{i+1}, R)$.
	\item $\mathcal{H} = \mathcal{H} + \dfrac{\log(g_i)}{d^{i+1}}$ 
	\item $(x_{i+1}, y_{i+1}) = (x'_{i+1}/g_{i}, y'_{i+1}/g_{i})$.
\end{enumerate}
  \item Return $\mathcal{H}$.
\end{enumerate}
\end{algorithm}

\begin{proposition}\label{p:algorithm_analysis}
Algorithm \ref{a:algorithm} computes $\mathcal{H}$, which satisfies
\[
\mathcal{H} = \sum_{n = 0}^{N-1}\frac{\Omega_{\Phi, 0}(\phi^{n}(P))}{d^{n+1}}
\]
and
\[
\left|\mathcal{H}_{\Phi, 0}(P) - \mathcal{H}\right| \leq \frac{\log{\left|\Res(F, G)\right|}}{(d-1)d^{N}}.
\]
\end{proposition}
\begin{proof}
Let $x_{j}, y_{j}, x'_{j}, y'_{j}, g_{j}, R$, and $\mathcal{H}$ be as defined in Algorithm \ref{a:algorithm}. We wish to show that
\begin{equation}\label{e:return_value}
\mathcal{H} = \sum_{n = 0}^{N-1}\frac{\Omega_{\Phi, 0}(\phi^{n}(P))}{d^{n+1}}
\end{equation}
and that
\begin{equation}\label{e:inequality}
\left|\mathcal{H}_{\Phi, 0}(P) - \sum_{n = 0}^{N-1}\frac{\Omega_{\Phi, 0}(\phi^{n}(P))}{d^{n+1}}\right| \leq \frac{\log{\left|\Res(F, G)\right|}}{(d-1)d^{N}}.
\end{equation}
By Lemma \ref{l:omega_bound}, we have
\begin{align*}
\left|\mathcal{H}_{\Phi, 0}(P) - \sum_{n = 0}^{N-1}\frac{\Omega_{\Phi, 0}(\phi^{n}(P))}{d^{n+1}}\right| &\leq \sum_{n = N}^{\infty}\frac{\left|\Omega_{\Phi, 0}(\phi^{n}(P))\right|}{d^{n+1}} \\
     &\leq \sum_{n = N}^{\infty}\frac{\log{\left|\Res(F, G)\right|}}{d^{n+1}} \\
     &= \frac{\log{\left|\Res(F, G)\right|}}{(d-1)d^{N}},
\end{align*}
which establishes \eqref{e:inequality}.

Recursively define a sequence $(a_{j}, b_{j})$ by $(a_{0}, b_{0}) = (x_{0}, y_{0})$ and
\[
(a_{j+1}, b_{j+1}) = \left(\frac{F(a_{j}, b_{j})}{\gcd(F(a_{j}, b_{j}), G(a_{j}, b_{j}))}, \frac{G(a_{j}, b_{j})}{\gcd(F(a_{j}, b_{j}), G(a_{j}, b_{j}))}\right)
\]
for $j = 0, \ldots, N-1$. It is clear by induction that $(a_{j}, b_{j}) \in \mathbb{Z}^{2}$ with $\gcd(a_{j}, b_{j}) = 1$ and that $\phi^{j}(P) = [a_{j}, b_{j}]$ for $j = 0, \ldots, N-1$. In particular, Lemma \ref{l:compute_omega} implies that
\[
\Omega_{\Phi, 0}(\phi^{j}(P)) = \log(\gcd(F(a_{j}, b_{j}), G(a_{j}, b_{j})))
\]
for $j = 0, \ldots, N-1$, so to establish \eqref{e:return_value}, it suffices to show that
\[
g_{j} = \gcd(F(a_{j}, b_{j}), G(a_{j}, b_{j}))
\]
for $j = 0, \ldots, N-1$. By using induction on $j$, we will show the stronger result that the following three equations hold for $j = 0, \ldots, N-1$:
\begin{enumerate}[label=(\alph*)]
  \item \label{item:a} $x_{j} \equiv a_{j} \pmod{R^{N - j}}; \quad y_{j} \equiv b_{j} \pmod{R^{N - j}}$.
	\item \label{item:b} $x'_{j+1} \equiv F(a_{j}, b_{j}) \pmod{R^{N-j}}; \quad y'_{j+1} \equiv G(a_{j}, b_{j}) \pmod{R^{N-j}}$.
	\item \label{item:c} $g_{j} = \gcd(F(a_{j}, b_{j}), G(a_{j}, b_{j}))$.
\end{enumerate}
For the base case $(j = 0)$, we have that \ref{item:a} and \ref{item:b} hold because by definition, $(a_{0}, b_{0}) = (x_{0}, y_{0})$ and
\[
x'_{1} \equiv F(x_{0}, y_{0}) \pmod{R^{N}}; \quad y'_{1} \equiv G(x_{0}, y_{0}) \pmod{R^{N}}.
\]
Moreover, Lemma \ref{l:compute_omega_modulo} and \ref{item:b} imply that
\[
\gcd(F(a_{0}, b_{0}), G(a_{0}, b_{0})) = \gcd(x'_{1}, y'_{1}, R),
\]
which establishes \ref{item:c}.

Now assume that \ref{item:a}, \ref{item:b}, and \ref{item:c} hold for $j = m-1$, for some fixed $m \leq N-1$. By \ref{item:b}, we have
\begin{align*}
x'_{m} &\equiv F(a_{m-1}, b_{m-1}) \pmod{R^{N-(m-1)}}, \\
y'_{m} &\equiv G(a_{m-1}, b_{m-1}) \pmod{R^{N-(m-1)}}.
\end{align*}
We know that $g_{m-1}$ divides $x'_{m}, y'_{m}$, and $R$ by definition, and \ref{item:c} implies that
\[
g_{m-1} = \gcd(F(a_{m-1}, b_{m-1}), G(a_{m-1}, b_{m-1}));
\]
it follows that
\begin{align*}
\frac{x'_{m}}{g_{m-1}} &\equiv \frac{F(a_{m-1}, b_{m-1})}{\gcd(F(a_{m-1}, b_{m-1}), G(a_{m-1}, b_{m-1}))} \pmod{R^{N-m}}, \\
\frac{y'_{m}}{g_{m-1}} &\equiv \frac{G(a_{m-1}, b_{m-1})}{\gcd(F(a_{m-1}, b_{m-1}), G(a_{m-1}, b_{m-1}))} \pmod{R^{N-m}},
\end{align*}
which establishes \ref{item:a} for $j = m$. Then \ref{item:b} for $j = m$ follows immediately from \ref{item:a} (for $j = m$) and from the definition of $x'_{m+1}, y'_{m+1}$. Finally, Lemma \ref{l:compute_omega_modulo} applied to \ref{item:b} when $j = m$ yields
\[
\gcd(F(a_{m}, b_{m}), G(a_{m}, b_{m})) = \gcd(x'_{m+1}, y'_{m+1}, R),
\]
which shows that \ref{item:c} holds for $j = m$.
\end{proof}

\begin{remark}
The modulus in Step (3) of Algorithm \ref{a:algorithm} decreases with each iteration of the loop. As a result, each successive iteration generally has a faster runtime than the previous one.
\end{remark}

\begin{remark}\label{r:resultant}
The error bound
\[
\left|\mathcal{H}_{\Phi, 0}(P) - \mathcal{H}\right| \leq \frac{\log{\left|\Res(F, G)\right|}}{(d-1)d^{N}}
\]
between the output $\mathcal{H}$ of Algorithm \ref{a:algorithm} and the true value of $\mathcal{H}_{\Phi, 0}(P)$ involves the quantity $\left|\Res(F, G)\right|$. Moreover, the runtime of our algorithm is determined by the time needed to perform calculations with numbers of size $O(\Res(F, G)^N)$. So to understand the accuracy and efficiency of Algorithm \ref{a:algorithm}, it is necessary to have a bound on the size of $\Res(F, G)$.

Recall (e.g., \cite[Section 2.4]{Silverman:2007}) that the resultant of two homogeneous polynomials of degree $d$ can expressed as the determinant of a certain $2d \times 2d$ matrix involving the coefficients of the polynomials. It follows that
\[
\left|\Res(F, G)\right| \leq (2d)!\norm{F, G}^{2d}.
\]
\end{remark}

\begin{remark}\label{r:factorization}
We can incorporate a factoring step in Algorithm \ref{a:algorithm}, as follows. Choose some ``small'' bound $B$ and factor $R = \Res(F, G)$ into primes as
\[
R = p_{1}^{e_{1}} \cdots p_{t}^{e_{t}} \cdot \tilde{R},
\]
with each $p_{i} \leq B$. Then $\mathcal{H}_{\Phi, 0}(P)$ is closely approximated by
\[
\widetilde{\mathcal{H}} + \sum_{i=1}^{t}\tilde{\lambda}_{\phi, p_{i}}(P),
\]
where the $\tilde{\lambda}_{\phi, p_{i}}(P)$'s are the modified local canonical heights defined by \eqref{eq:modified_local_height}, and where $\widetilde{\mathcal{H}}$ denotes the output of Algorithm \ref{a:algorithm} with $R$ replaced by $\tilde{R}$ in Step (3). As previously noted, there are efficient algorithms in the literature for computing $\tilde{\lambda}_{\phi, p_{i}}(P)$; alternatively, we can compute $\tilde{\lambda}_{\phi, p_{i}}(P)$ by running Algorithm $\ref{a:algorithm}$ with $R$ replaced by $p_{i}^{e_{i}}$ in Step (3).

More generally, if $R$ factors into pairwise relatively prime integers as
\[
R = \tilde{R}_{1}\tilde{R_{2}}\cdots\tilde{R}_{t},
\]
then $\mathcal{H}_{\Phi, 0}(P)$ is approximated by the sum
\[
\widetilde{\mathcal{H}}_{1} + \widetilde{\mathcal{H}}_{2} + \cdots + \widetilde{\mathcal{H}}_{t},
\]
where $\widetilde{\mathcal{H}}_{i}$ denotes the output of Algorithm \ref{a:algorithm} with $R$ replaced by $\tilde{R}_{i}$ in Step~(3). This might be useful, for instance, if we can factor $R$ as $R = \tilde{R}_{1}\tilde{R_{2}}$, where $\tilde{R}_{1}$ and $\tilde{R}_{2}$ are large, composite, and relatively prime.
\end{remark}

\section{Numerical Examples}\label{s:examples}

In this section, we give some examples illustrating the use of Algorithm~\ref{a:algorithm}. The most novel and useful aspect of our algorithm is that it does not require that we factor $R = \Res(F, G)$. By Remark \ref{r:resultant}, we expect that $R \approx (2d)!\norm{F, G}^{2d}$ for many morphisms $\phi$, so Algorithm \ref{a:algorithm} is particularly advantageous if $d$ is of moderate size, or if $F$ and $G$ have large coefficients. We give two examples demonstrating each of these scenarios.

We also note that the current implementation in Sage \cite{Sage} for computing $\hat{h}_{\phi}(P)$ uses the decomposition \eqref{eq:modified_local_heights} of $\hat{h}_{\phi}(P)$ into local canonical heights. In particular, one of the steps in this algorithm requires the factorization of $R$, rendering the algorithm completely impractical for morphisms $\phi$ for which $R$ is very large. Algorithm \ref{a:algorithm} can be used to compute $\hat{h}_{\phi}(P)$ in these cases.

In the examples that follow, we continue to employ the same notation as in the previous sections, including the notation in Algorithm \ref{a:algorithm}.

\begin{example}\label{e:1}
Define ``random'' polynomials
\begin{align*}
\tau_{1}(z) &= 3z^{80} + z^{79} + 4z^{78} + z^{77} + 5z^{76} + \cdots + 9z^{2} + 3z + 7 \\
\tau_{2}(z) &= 2z^{80} + 7z^{79} + z^{78} + 8z^{77} + 2z^{76} + \cdots + 6z^{2} + 9
\end{align*}
such that the coeffient of $z^{81-i}$ in $\tau_{1}$ (respectively $\tau_{2}$) equals the $i$th digit of $\pi$ (respectively $e$), and set
\[
\sigma(z) = \frac{\tau_{1}(z)}{\tau_{2}(z)}.
\]
Let $\phi: \mathbb{P}_{\mathbb{Q}}^{1} \rightarrow \mathbb{P}_{\mathbb{Q}}^{1}$ be the morphism of degree $d = 80$ induced by the rational map $\sigma$. We calculate the canonical height of the point $P = [-5, 1]$ relative to $\phi$.

Taking $\Phi = [F, G]$ to be a lift of $\phi$, where $F$ and $G$ are the respective degree $80$ homogenizations of $\tau_{1}$ and $\tau_{2}$, we compute
\begin{align*}
\Res(F, G) &= 516438964415067184645\cdots303541485376492059392 \\
           &= 2^{8} \cdot 3^{2} \cdot R' \\ 
           &\approx 2^{653}
\end{align*}
for some large $R'$. In particular, $\Res(F, G)$ is over $650$ bits and is thus time-consuming to factor into primes.

Using Algorithm \ref{a:algorithm} with $N = 50$, we compute
\[
g_{0} = 36, \quad g_{1} = 2, \quad g_{2} = 12,
\]
and
\[
g_{i} = \begin{cases}
        2 & \text{if} \ i \equiv 1 \pmod{2} \\
        4 & \text{if} \ i \equiv 0 \pmod{2}
        \end{cases}
\]
for $3 \leq i \leq 49$, so that
\begin{align*}
\mathcal{H}_{\Phi, 0}(P) &\approx \sum_{i=0}^{49}\frac{\log(g_{i})}{d^{i+1}} \\
                         &\approx 0.044907161659276960113044136254.
\end{align*}
By Proposition \ref{p:algorithm_analysis}, the error in this approximation of $\mathcal{H}_{\Phi, 0}(P)$ (prior to truncating the decimal expansion) is at most
\[
\frac{\log{\left|\Res(F, G)\right|}}{(d-1)d^{N}} < 10^{-94}.
\]
Note that for this computation, we must work with numbers modulo $\Res(F, G)^{50}$, which is approximately $32674$ bits.

Using a close variant of the algorithm in \cite[exercise 5.29]{Silverman:2007} with $50$ iterations, we compute
\[
\mathcal{H}_{\Phi, \infty}(P) \approx -0.013757185585214127675440651473.
\]

We also have
\[
h(P) = \log(5) \approx 1.6094379124341003746007593332.
\]

So by Proposition \ref{p:canonical_height_formula}, we have
\begin{align*}
\hat{h}_{\phi}(P) &= h(P) - \mathcal{H}_{\Phi, \infty}(P) - \mathcal{H}_{\Phi, 0}(P) \\
                  &\approx 1.5782879363600375421631558484.
\end{align*}
\end{example}

\begin{example}
This example is similar to Example \ref{e:1}, except in this case we use a point whose canonical height relative to our chosen morphism is very small.

Define ``random'' polynomials
\[
\tau_{1}(z) = \sum_{i=0}^{65}a_{i}z^{65-i}, \quad \tau_{2}(z) = \sum_{i=0}^{65}b_{i}z^{65-i}
\]
by
\[
a_{i} = \begin{cases}
        -i & \text{if} \ i \ \text{is prime} \\
        1 & \text{if} \ i \ \text{is not prime}
        \end{cases}, \quad b_{i} = \begin{cases}
        1 & \text{if} \ 0 \leq i \leq 33 \\
        -1 & \text{if} \ 34 \leq i \leq 65
        \end{cases}
\]
and let 
\[
\sigma(z) = \frac{\tau_{1}(z)}{\tau_{2}(z)}.
\]
Let $\phi: \mathbb{P}_{\mathbb{Q}}^{1} \rightarrow \mathbb{P}_{\mathbb{Q}}^{1}$ be the degree $d = 65$ morphism induced by $\sigma$, and let $\psi = [F, G]$ be the lift of $\phi$ obtained by taking the degree $65$ homogenizations of $\tau_{1}$ and $\tau_{2}$. We have
\begin{align*}
\Res(F, G) &= 201910883195612036622\cdots662564775325296900059 \\
           &= 3^{3} \cdot 19 \cdot R' \\
           &\approx 2^{433}
\end{align*}
for some large $R'$; at over $430$ bits, $\Res(F, G)$ is difficult to factor.

We calculate the canonical height of the point $P = [0, 1]$ relative to $\phi$. The orbit of $P$ under $\phi$ starts as
\[
[0, 1] \mapsto [-1, 1] \mapsto [1, 0] \mapsto [1, 1] \mapsto [-453, 2] \mapsto \cdots,
\]
so one might expect that $\hat{h}_{\phi}(P)$ is small. From Algorithm \ref{a:algorithm} with $N = 50$, which gives an approximation of $\mathcal{H}_{\Phi, 0}(P)$ that is accurate to within $10^{-89}$, we compute
\[
g_{0} = 1, g_{2} = 513, g_{2} = 1, g_{3} = 1, \ldots, g_{46} = 19, g_{47} = 1, g_{48} = 1, g_{49} = 27;
\]
each $g_{i} \in \{1, 19, 27, 513\}$, and the sequence of $g_{i}$'s is periodic with period $20$ (at least for the first $50$ terms in the sequence). Note also that $513 = 3^{3} \cdot 19$ really does divide $\Res(F, G)$. So
\begin{align*}
\mathcal{H}_{\Phi, 0}(P) &\approx \sum_{i=0}^{49}\frac{\log(g_{i})}{d^{i+1}} \\
                         &\approx 0.0014769884100219430907588636039.
\end{align*}
We also compute
\[
\mathcal{H}_{\Phi, \infty}(P) \approx -0.0014773310580301870814703316397
\]
and
\[
h(P) = \log(1) = 0.
\]
Hence, 
\begin{align*}
\hat{h}_{\phi}(P) &= h(P) - \mathcal{H}_{\Phi, \infty}(P) - \mathcal{H}_{\Phi, 0}(P) \\
                  &\approx 0.00000034264800824399071146803578925.
\end{align*}

\end{example}

\begin{example}
Consider the family of degree $2$ morphisms $\phi_{a}: \mathbb{P}_{\mathbb{Q}}^{1} \rightarrow \mathbb{P}_{\mathbb{Q}}^{1}$, where $\phi_{a}$ is induced by the rational map
\[
z \mapsto \frac{z^{2} + z + 1}{z^{2} + az + 2}, \quad a \in \mathbb{Z}.
\]
A lift $\psi_{a} = [F_{a}, G_{a}]$ for $\phi_{a}$ is given by
\[
F_{a}(X, Y) = X^{2} + XY + Y^{2}, \quad G_{a}(X, Y) = X^{2} + aXY + 2Y^{2},
\] 
and the corresponding resultant is
\[
\Res(F_{a}, G_{a}) = a^{2} - 3a + 3,
\]
which is difficult to factor when $a$ is large. For instance, taking 
\[
a = 314159265358979323846\cdots964462294895493038196
\]
to equal the number formed by the first $201$ digits of $\pi$, we have
\[
\Res(F_{a}, G_{a}) = 3 \cdot 7 \cdot 61 \cdot R',
\]
where $R' \approx 2^{1321}$. For this value of $a$ and $P = [1, 1]$, we calculate $\hat{h}_{\phi}(P)$.

Using $N = 50$ terms in Algorithm \ref{a:algorithm}, which allows for an approximation with error less than $10^{-12}$, we compute
\[
g_{0} = 3, g_{1} = 1, g_{2} = 1, g_{3} = 3, \ldots, g_{46} = 3, g_{47} = 1, g_{48} = 3, g_{49} = 1;
\]
each $g_{i} \in \{1, 3\}$, and there is no discernible repeating pattern in the sequence of $g_{i}$'s. This gives the approximation
\begin{align*}
\mathcal{H}_{\Phi, 0}(P) &\approx \sum_{i=0}^{49}\frac{\log(g_{i})}{2^{i+1}} \\
                         &\approx 0.62900702. 
\end{align*}
We also calculate
\[
\mathcal{H}_{\Phi, \infty}(P) \approx -308.06749879, 
\]
and
\[
h(P) = \log(1) = 0.
\]
Therefore,
\begin{align*}
\hat{h}_{\phi}(P) &= h(P) - \mathcal{H}_{\Phi, \infty}(P) - \mathcal{H}_{\Phi, 0}(P) \\
                  &\approx 307.43849177. 
\end{align*}
\end{example}

\begin{example}
Let $\phi_{a}: \mathbb{P}_{\mathbb{Q}}^{1} \rightarrow \mathbb{P}_{\mathbb{Q}}^{1}$ be the much-studied family of degree $2$ morphisms induced by the rational map
\[
z \mapsto az + \frac{1}{z}, \quad a \in \mathbb{Z}.
\]
See for example \cite{Manes}. We can lift $\phi_{a}$ to $\psi_{a} = [F_{a}, G_{a}]$, where
\[
F_{a}(X, Y) = aX^{2} + Y^{2}, \quad G_{a}(X, Y) = XY.
\]
The resultant of $F_{a}$ and $G_{a}$ is
\[
\Res(F_{a}, G_{a}) = a,
\]
which is time-consuming to factor for large $a$.

As an example, we compute the canonical height of the point $P = [a, 1]$ relative to $\phi_{a}$ in the case where
\begin{align*}
a &= 123018668453011775513\cdots419597459856902143413 \\
  &\approx 2^{768}
\end{align*}
is RSA-768 -- a 768-bit RSA modulus which took a team of researchers over 2 years to factor with the number field sieve (c.f. \cite{RSA}).

Running Algorithm \ref{a:algorithm} to $N = 50$ terms, which gives an approximation with error bounded above by $10^{-12}$, we get $g_{0} = 0$, $g_{1} = a$, and $g_{i} = 1$ for $2 \leq i \leq 49$. This yields the approximation
\begin{align*}
\mathcal{H}_{\Phi, 0}(P) &\approx \frac{\log(a)}{2^{2}} \\
                         &\approx 133.0260806. 
\end{align*}
We also calculate
\[
\mathcal{H}_{\Phi, \infty}(P) \approx -532.1043224 
\]
and
\[
h(P) = \log(a) \approx 532.1043224. 
\]
Our computation seems to indicate that $\left|\mathcal{H}_{\Phi, \infty}(P)\right| = \log(a)$. A careful analysis of the archimedean term reveals that in fact $\left|\mathcal{H}_{\Phi, \infty}(P)\right| > \log(a)$, but the difference is minuscule and so cannot be detected with just $N = 50$ terms. 

It follows that
\begin{align*}
\hat{h}_{\phi}(P) &= h(P) - \mathcal{H}_{\Phi, \infty}(P) - \mathcal{H}_{\Phi, 0}(P) \\
                  &\approx 931.1825642. 
\end{align*}
\end{example}

\bibliographystyle{plain} 
\bibliography{canonical_height_algorithm_bibliography}

\end{document}